\documentclass[12pt]{amsart}
\usepackage{amssymb} 

\newtheorem{thm}{Theorem}[section]
\newtheorem{cor}[thm]{Corollary}
\newtheorem{lem}[thm]{Lemma}
\newtheorem{prop}[thm]{Proposition}
\newtheorem{exam}[thm]{Example}

\theoremstyle{definition}
\newtheorem{defn}[thm]{Definition}
\newtheorem{rem}[thm]{Remark}

\newtheorem{que}[thm]{Question}

\numberwithin{equation}{section}

\usepackage[colorlinks]{hyperref}

\begin{document}

\title{On  the structure of  $LC$-nilpotent   groups}
\author{M. Amiri,  I. Kashuba, I. Lima \\
}
\footnotetext{E-mail Address: {\tt m.amiri77@gmail.com;\,  	ikashuba@gmail.com\, igor.matematico@gmail.com } }
\date{}
\maketitle

\begin{quote}
{\small \hfill{\rule{13.3cm}{.1mm}\hskip2cm}
\textbf{Abstract.} {For a finite group $G$, let $LC(G)$ be the subgroup generated by elements $x$ such that, for all $y \in G$ and all integers $n$, the order of $x^n y$ divides the least common multiple of the orders of $x$ and $y$. This subgroup is a nilpotent characteristic subgroup of $G$. 
In this article, among other results, we show that a finite solvable group $G$ admits an $LC$-nilpotent series if and only if $G$ does not contain any $2$-Frobenius section  of type $(p, q, p)$. As a consequence of this theorem, we conclude that the algebraic system consisting all $LC$-nilpotent groups forms a variety. Finally, we answer to Question 3.7 from \cite{mohsen} in a more general case.}

}\\
\vspace{-3mm}\hfill{\rule{13.3cm}{.1mm}\hskip2cm}
\end{quote}


\maketitle

\section{Introduction}\label{sec1}

 Let $G$ be a periodic group and let $LCM(G)$ be the set of all $x \in G$ such that $o(x^n y)$ divides the least common multiple of $o(x^n)$ and $o(y)$ for all $y \in G$ and all integers $n$. The subgroup generated by $LCM(G)$ is denoted by $LC(G)$. A group $G$ is said to be an $LCM$-group if $G = LCM(G)$.

In \cite{mohsen} the authors introduced the $LCM$-groups and $LC$-series. They prove many results about these groups, for instance
\begin{thm}   Let $G$  be a locally finite  group.
Then 
 $  LC(G)$ is a locally nilpotent subgroup of $G$.
\end{thm}

Therefore, as a natural continuation, in this paper we explore the structure  of   $LC$-nilpotent groups defined in \cite{mohsen}.

One of the main motivations for studying the behavior, in periodic groups, of the order of elements on the structure of the group, arises from the celebrated Burnside Problem. Here our focus is to understand how the finite order of the product of the elements can determine the structure of the group. In \cite{mohsen}, it was proven, for a periodic group $G$, that $LCM(G)$ is an $\operatorname{Aut}(G)$-invariant subgroup of $G$ (see Lemma 2.4 in \cite{mohsen}). The main ingredient of our proofs is the fact that the behavior of the $p$ elements in $LCM(G)$ is enough to determine all the information about $LCM(G)$.

Let $G$ be a group.
As defined in \cite{mohsen}, set $LC_1(G)=LC(G)$ and for $i=2,3,\ldots$, define
$LC(\frac{G}{LC_{i-1}(G)})=\frac{LC_{i}(G)}{LC_{i-1}(G)}$.
We say the group $G$ is an $LC$-nilpotent group whenever 
there exists a finite $LC$-series 
\begin{equation}\label{eq11}
LC_0(G)=1\leq LC_1(G)\leq LC_2(G)\leq\ldots\leq LC_k(G)=G
\end{equation}
such that $\frac{LC_i(G)}{LC_{i-1}(G)}$ is a nilpotent group for all $i=1,2,\ldots,k$. In this case, the $LC$-series (\ref{eq11}) is referred to as $LC$-nilpotent series for $G$, and $G$ is called a group of $LC$-class $k$ whenever $LC_{k-1}(G)<LC_k(G)$.
Also, we say $G$ is a group of  maximal $LC$-class whenever $\frac{LC_i(G)}{LC_{i-1}(G)}$ is a prime number for each $i=1,2,\ldots,k$.  It is a simple exercise to show that any finite $p$-group is a $LC$-nilpotent group.

  Recall that a $2$-Frobenius group $G=ABC$, where $A$ and $AB$ are normal
subgroups of $G$, and $AB$ and $BC$ are Frobenius group with kernel $A$, $B$ and
complements $B$, $C$, respectively. 
 In particular, $G$ is said to be a $2$-Frobenius of type $(p, q, p) $ if both $\frac{G}{AB}$ and $A$ are $p$-groups and $B$ is a $q$-group. A section of $G$ is a quotient group $H/K$, where $K \trianglelefteq H \leq G$.
In this paper, we prove the following theorem:
\begin{thm}\label{maa1}
    Let $G$ be a finite group. The following conditions are equivalent:
    \begin{enumerate}
        \item $G$ is an $LC$-nilpotent group.
        \item No section of $G$ contains a $2$-Frobenius subgroup of type $(p, q, p)$.
        \item For all $x, y \in G$, we have $LC(\langle x, y \rangle) \neq 1$.
    \end{enumerate}
\end{thm}

In \cite{mohsen} the following question was posed

\begin{que}[Question 3.7]
    What is the set of all LC-nilpotent groups of class two?
    \label{question1}
\end{que}

The main result of our work provides a more general answer to this question. As another consequence of the above theorem, we observe that the algebraic system consisting all $LC$-nilpotent groups forms a variety. More precisely:
 \begin{cor} 
     Let $\mathfrak{L}$ be the set of all $LC$-nilpotent finite groups.
     Then every subgroup and quotient of a group in $\mathfrak{L}$ also belongs to $\mathfrak{L}$.
     Also, for  $G,H\in \mathfrak{L}$, we have $G\times H\in \mathfrak{L}$.
 \end{cor}
 We have the following chain of inclusions:
\[
\mathit{Nil} \subset \mathit{Sup} \subset \mathfrak{L} \subset \mathit{Sol}
\]
where $\mathit{Nil}$ denotes the set of all finite nilpotent groups, $\mathit{Sup}$ the set of all finite supersolvable groups, and $\mathit{Sol}$ the set of all finite solvable groups. Furthermore, all inclusions are proper.

For a finite group $G$, the vertices of the prime graph $\Gamma(G)$ are the primes that divide $|G|$, and two vertices $p$ and $q$ are connected by an edge   if and only if there is an element of order $pq$ in $G$.
For any prime divisor $p$ of $|G|$, let 
$deg(p)=\sum_{q\mid |G|, \ q\neq p}e(p,q).$

Also, we prove the following theorems about some special class of $LC$-nilpotent groups:
 
 \begin{thm}
  Let $G$ be a finite group. Then $G$ is an $LC$-nilpotent group of maximal $LC$-class $k$ if and only if $k=2$ and 
  $G$ is a non-abelian group of order $pq$ where $p < q$ are prime numbers.
 \end{thm}

\begin{thm} 

     Let $G$ be a finite non-abelian  $p$-group of $LC$-class  $k>1$. Then   for each  $ 0\leq i\leq  \ldots\leq  k-1$,   $\frac{LC_{i+1}(G)}{LC_i(G)}$ is cyclic if and only if  $k=2=p$, and  
  $G \in \{D_{2^n}, SD_{2^n}, Q_{2^m}\}$, where $n \geq 3$ and $m \geq 4$.
 \end{thm}

  \begin{thm} 
     Let $G$ be a finite group. Then $G$ is an $LC$-nilpotent group of $LC$-class $k$ such that $\frac{LC_{i+1}(G)}{LC_i(G)}$ is cyclic for all $i = 0, \ldots, k-1$ if and only if $G = LC(G)H$, where $H$ is a cyclic subgroup of $G$.
   \end{thm}

In what follows, we adopt the notation established in Isaacs' book on finite groups \cite{I}.

\section{$LCM_p(G)$ of \ periodic \ groups}

For a $p$-group $G$ and an integer $n$, we denote the sets $\langle\{x\in G| \ x^{p^{n}}=1\}\rangle$ and $\langle\{x^{p^n}\in G| \ x\in G\}\rangle$ by $\Omega_{n}(G)$ and $\mho_{n}(G)$, respectively. 
Let $CP2$ be the class of finite groups $G$ such that $o(xy)\leq max\{o(x), o(y)\}$ for all $ x\neq y \in G$ \cite{Deb}. 
 
We shall need the following results.

\begin{thm}(Theorem D in \cite{Deb}) \label{a} A finite group $G$ is contained in $CP2$ if and only if one of the following statements holds:

\begin{enumerate}
  \item $G$ is a $p-$group and $\Omega_{n}(G)=\{x\in G\ | \ x^{p^{n}}=1\}$ for all integers $n$.
  \item $G$ is a Frobenius group of order $p^{\alpha}q^{\beta}$, $p<q$, with kernel $F(G)$ of order $p^{\alpha}$ and cyclic complement.
\end{enumerate}

\end{thm}
\begin{thm}(Theorem 2.6 in \cite{mohsen})\label{12}   Let $G$  be a finite  group. 
Then $G$ is an $LCM$-group if and only if $G$ is a nilpotent group and each Sylow subgroup of $G$ is in $CP2$.

\end{thm}

From Lemma 2.4 of \cite{mohsen}, if $G$ is a periodic group, then $LCM(G)$ is an $\operatorname{Aut}(G)$-invariant subgroup of $G$.

Let $G$ be a periodic group and $p$ a prime number. We denote the set of all $p$-elements $x\in G$ such that
$o(hy)\mid lcm(o(h),o(y))$ for any $p$-element $y$ of $G$ and all $h\in \langle x\rangle$ by $LCM_p(G)$.
\begin{lem}\label{ces}
Let $G$ be a periodic group and $p$ a prime number. 

(i) For any $\sigma\in \operatorname{Aut}(G)$, we have $LCM_p(G)^{\sigma}=LCM_p(G)$.

(ii) The subgroup generated by $LCM_p(G)$ is a $p$-group.

\end{lem}
\begin{proof}{    
(i) Let $x\in LCM_p(G)$ and let $\sigma\in \operatorname{Aut}(G)$.
Let $y$ be a $p$-element of $G$, and let  $h\in \langle x\rangle$.
Then there exists $p$-element $z\in G$ such that $\sigma(z)=y$.
Then $$o(\sigma(h)y)=o(\sigma(hz))=o(hz)\mid lcm(o(h),o(z))=lcm(o(\sigma(h)),o(y)).$$

(ii)  Let $x\in LCM_p(G)$. 
Let $S$ be a subset of $x^G$ such that  $F:=\langle x^G\rangle=\langle S\rangle$ but $F\neq \langle T\rangle$ for any proper subset $T$ of $S$.
Any $w\in F$ is just a finite sequence $w=s_{1}\ldots s_{r}$ whose entries $ s_{1},\ldots ,s_{r}$ are elements of $S\cup S^{-1}$. The integer $r$ is called the length of the element $w$ and its norm $|w|$ with respect to the generating set $S$ is defined to be the shortest length of $w$ over $S$.  
  Let $y\in F.$ We claim that $o(y)\mid o(x)$.
We proceed by induction on $|y|$.   The case $|y|=0$ is trivial.
Suppose that the result is true for all $a\in F$ with $|a|<|y|$. There are $s_1,\ldots,s_k\in S$, $\epsilon_i\in\{1,-1\}$ and positive integers $n_1,\ldots,n_k$ such that
 $y=(s_1)^{\epsilon_1n_1}\ldots(s_k)^{\epsilon_kn_k}$. 
 We may assume that $n_1>0$.
By the induction hypothesis, we have $o((s_1)^{\epsilon_1(n_1-1)}\ldots(s_k)^{\epsilon_kn_k})\mid o(x)$. 
From part (i) $x^g,(x^g)^{-1}\in LCM_p(G)$ for all $g\in G$, thus we have $S\subseteq LCM(H_p(G), A)$, and therefore $s_1\in LCM_p(G)$. Consequently,
\begin{eqnarray*}
o(y)&=&o(s_1(s_1)^{\epsilon_1 (n_1-1)}(s_k)^{\epsilon_2 n_2}\ldots(s_k)^{\epsilon_k n_k})\\&\mid& lcm(o(s_1), o((s_1)^{\epsilon_1 (n_1-1)}(s_2)^{\epsilon_2 n_2}\ldots(s_k)^{\epsilon_k n_k}))\mid\\&\vdots&\\&\mid&
lcm(o(s_1),o(x))\\&=&o(x).
\end{eqnarray*}
\noindent Let $x_1,x_2\in LCM_p(G)$.
Since $\langle x_1^G,x_2^G\rangle=\langle x_1^G\rangle\langle x_2^G\rangle$, we have $o(x_1x_2)\mid lcm(o(x_1),o(x_2))$. It follows that $\langle LCM_p(G)\rangle$ is a $p$-group. 
}\end{proof}
The following proposition shows that $LCM_p(G)$ is a subset of $LCM(G)$ whenever $G$ is a periodic group.
\begin{prop}\label{equ}
Let $G$ be a periodic group and let $g\in G$ be a $p$-element. Then $g\in LCM_p(G)$ if and only if $g\in LCM(G)$.
\end{prop} 
 \begin{proof}
 If $g\in LCM(G)$, then by definition of $LCM(G)$, we have
  $o(hy)\mid lcm(o(h),o(y))$ for any $p$-element $y$ of $G$ and all $h\in \langle g\rangle$, thus $g\in LCM_p(G)$. 
  Suppose $g\in LCM_p(G)$,
then for all $p$-elements $y\in G$, we have $o(hy)\mid lcm(o(h),o(y))$ for any    $h\in \langle g\rangle$.
 Let $z\in G$ and let $h\in \langle g\rangle$, then $o(z)=p^ka$ where $p\nmid a$.
 There exist $y,v\in G$ such that $z=yv=vy$ where $o(y)=p^k$ and
 $o(v)=a$.
Then
$(zh)^{a}=h^zh^{z^2}\ldots h^{z^a}z^a=h^zh^{z^2}\ldots h^{z^a}y^a$.
From Lemma \ref{ces}, $ \langle h^G\rangle$ is a normal $p$-subgroup of $G$, thus $\langle h^G\rangle\langle y\rangle$ is a $p$-group. Hence 
$h^{z^i}\ldots h^{z^a}z^a$ is a $p$-element for $i=1,\ldots,a$.
Consequently,  
\begin{eqnarray*}
o((zh)^a)=o(h^zh^{z^2}\ldots h^{z^a}y^a)&\mid& lcm(o(h^z),o(h^{z^2}\ldots h^{z^a}y^a))\\&\mid&
lcm(o(h^z),o(h^{z^2}),o(h^{z^3}h^{z^4}\ldots h^{z^a}y^a))\\&\mid&
 \vdots
\\&\mid& lcm(o(h^z),\cdots, o(h^{z^a}),o(y))\\&=&
lcm(o(h),o(y)).
\end{eqnarray*}
Then $$o(zh)\mid a \cdot lcm(o(h),o(y))= lcm(o(h),a\cdot o(y))=lcm(o(h),o(z)),$$ and thus $g\in LCM(G)$. 
 \end{proof}
 
\begin{cor}\label{42}
Let $G$ be a finite group and let $P\in \operatorname{Syl}_p(\operatorname{Fit}(G))$ and $Q\in \operatorname{Syl}_p(G)$.
If $Q\in CP2$, then $P\subseteq LCM(G)$.
\end{cor}
If $G$ is a $p$-group, we can find an upper bound on the $LC$-class   of $G$ in terms of its nilpotency classes. To do this we need the following simple lemma.
 \begin{lem}\label{zp}
 Let $G$ be a finite $p$-group.
Then $Z_{p-1}(G)\leq LCM(G)$.
 \end{lem}
\begin{proof}
Let $x\in Z_{p-1}(G)$ and $y\in G$.
Let $H=\langle x,y\rangle$, and let $z\in \langle x\rangle$.
Since the nilpotency class of $H$ is less than $p$, $H$ is a regular group. Therefore $o(zy)\mid lcm(o(z),o(y))$.
Since $y$ is an arbitrary element of $G$, we deduce that $x\in LCM(G)$.
\end{proof} 
\begin{proof}
Let $x\in P$ and $y$ be a $p$-element  of order $p^m$. There exists $g\in G$ such that $y\in Q^g$. Since $P\leq Q^g$, we have $x^n,y\in Q^g$ for all integers $n$.
It follows from $Q^g\in CP2$ that
$o(x^ny)\mid lcm(o(x^n),o(y))$.
From Proposition \ref{equ}, $x\in LCM(G)$.
Then $P\subseteq LCM(G)$.
\end{proof}
  In the  theorem  we  establish a connection between the nilpotency class of a finite $p$-group and the $LC$-nilpotency class, for $p$ a prime number. 
\begin{thm}
 Let $G$ be a finite $p$-group of nilpotency class $t$.
Then $G$ is an $LC$-nilpotent group of $LC$-class at most $\lfloor t/(p-1)\rfloor+1$.
 \end{thm}
\begin{proof}
We proceed by induction on $t$.
If $t< p$, then
by Lemma \ref{zp}, $G=LCM(G)$, and so $G$ is an $LC$-nilpotent group of $LC$-class $\lfloor t/(p-1)\rfloor+1=1$. Suppose that $t\geq p$.
From Lemma \ref{zp}, $Z_{p-1}(G)\leq LC(G)$.
Then the nilpotency class of  $\frac{G}{LC(G)}$  is at most 
$t-p+1$. By the induction hypothesis, $\frac{G}{LC(G)}$ is an $LC$-nilpotent group of $LC$-class at most  $\lfloor (t-p+1)/(p-1)\rfloor+1$.
Therefore $G$ is a $LC$-nilpotent group of $LC$-class at most $\lfloor t/(p-1)\rfloor+1$.
\end{proof}
Let $G$ be a finite group, we say $G$ is a minimal $NLCM$-group whenever $G$ is not a $LCM$-group, but
all proper sections  of $G$ are $LCM$-group.  
From the following theorem, one we can  see that the $NLCM$-group has restricted structure.
We denote by $c(p)$ and $p^{n(p)}$, 
respectively, the maximum class and order of a $2$-generator finite group of 
exponent $p$.

In accordance with A. Mann \cite{man22}, a $p$-group $G$ is called a $P_i$-group (for $i = 1, 2, 3$) if $G$, as well as all sections of $G$, satisfy the $i$-th statement among the following:
\begin{enumerate}
    \item Each element of $\Omega_n(G)$ is a $p$-th power.
    \item Each element of $\mho_n(G)$ has order $p^n$ (at most).
    \item $|\mho_n(G)| = |G : \Omega_n(G)|$.
\end{enumerate}

Moreover, a $p$-group $G$ is called a $P$-group if $G$ and all its sections satisfy all three statements.

The following auxiliary result is used to prove Theorem \ref{min}.

\begin{lem}\label{Mann}(Theorem 6 \cite{man22})
Let $G$ be a minimal non-$P_2$-group. Then:
\begin{enumerate}
    \item[(a)] $G$ can be generated by two elements of order $p$; $\exp G = p^2$.
    \item[(b)] $G$ contains a maximal subgroup of exponent $p$. In particular, $\exp \Phi(G) = p$; moreover, $\Phi(G) = G' = Z_{p-1}(G)$.
    \item[(c)] $Z(G) = \mho_1(G)$ has order $p$.
    \item[(d)] $G$ is a minimal non-$P$ group and a $P_1$-group.
    \item[(e)] Each proper section of $G$ has smaller class than $G$.
    \item[(f)] $cl G \leqslant c(p) + 1$, $|G| \leqslant p^{n(p)+1}$.
\end{enumerate}
\end{lem}

\begin{thm}\label{min}
Let $G$ be a finite $NLCM$-group.
\begin{itemize}
\item[(a)] If $G$ is not a $p$-group, then 
$G$ is a minimal non-nilpotent group.
\item[(b)] If $G$ is a $p$-group, then 
\begin{enumerate}
\item $G$ can be generated by two elements of order p; $exp(G)=p^2$. 
    \item $G$ has a maximal subgroup $H$ such that $exp(H)=p$ and $G=H\rtimes \langle u\rangle$ for some $u\in G$ of order $p$, moreover, $\varPhi(G)=G'=Z_{c-1}(G)$. 
    \item $Z(G)=\mho_1(G)$ has order $p$.

    \item Each proper section of $G$ has smaller nilpotency  class than  $G$. 

    \item $cl(G)<c(p)+1$, $|G|<p^{n(p)+1}.$

\end{enumerate}
\end{itemize}
\end{thm}
\begin{proof}
(a)  From Corollary 2.15  of \cite{mohsen}, $G$ is a minimal non-nilpotent group.

(b) It follows by Lemma \ref{Mann}.
\end{proof}
\noindent We need the following Lemma to prove our first main result.

\begin{lem}\label{norm1}
Let $G$ be a finite group and let
$H=\langle a\rangle$ be a normal cyclic $p$-subgroup of $G$. Then $H\subseteq LCM(G)$.
\end{lem}
\begin{proof}
  Let $h\in H$.  Let  $y\in G$, be a $p$-element, and let $$s=\max\{o(h),o(y)\} =lcm(o(h),o(y)).$$

  Let $\Gamma_i(\langle h,y\rangle)$ be   the $i$-th member of the lower central series of  $\langle h,y\rangle$.
  Then  by the Hall–Petrescu formula we have $$(hy)^{s}=y^{s}h^sc_2^{(_2^s)}c_3^{(_3^s)}\ldots c_s^{(_s^s)}$$
     where $c_i\in \Gamma_i(\langle h,y\rangle)$ for $i=2,\ldots,s$.
Since $c_i\in \mho_1(H)$, for all $i=2,\ldots,s$ it follows
$c_i^s=1$ for all $i=2,\ldots,s-1$.
Since $\langle a\rangle \langle y\rangle$ is a metacyclic $p$-group,  $c_s=1$.
It follows that $o(hy)\mid s$, thus  $o(hy)\mid lcm(o(h),o(y))$. Therefore  $H\subseteq LCM(G)$.
\end{proof}
 
   \begin{cor}\label{dih}
      If $G\in\{ D_{2^n},SD_{2^n},Q_{2^m}\}$ where $n\geq 3$ and $m\geq 4$, then   $LC(G)$ is the  cyclic subgroup of index $2$ of $G$. 
  \end{cor}
  
Let $G$ be a finite group and $N$ a normal subgroup of $G$. In general $\frac{LC(G)N}{N}$ is not a subgroup of $LC\left(\frac{G}{N}\right)$, but in some cases the inclusion holds. The following proposition demonstrates that the product of the $LC$-subgroups of two finite groups form a subgroup of the $LC$-subgroup of their direct product.  

\begin{prop}\label{pro}
Let $G$ and $H$ be two periodic groups. Then,
\[
LCM(G) \times LCM(H) \subseteq LCM(G \times H).
\]
Moreover, if $\gcd(\exp(G), \exp(H)) = 1$, then
\[
LCM(G) \times LCM(H) = LCM(G \times H).
\]
\end{prop}

\begin{proof}
Let $g \in \langle g_1 \rangle \subseteq LCM(G)$ and $h \in \langle h_1 \rangle \subseteq LCM(H)$.  
For any $x \in G$ and $y \in H$, since $H \cap G = 1$, we have
\[
o(xy) = lcm(o(x), o(y)) \quad \text{and} \quad o(gh) = lcm(o(g), o(h)).
\]
Furthermore,
\begin{align*}
o(ghxy) &\mid lcm(o(gx), o(hy)) \\
&\mid lcm(o(g), o(x), o(h), o(y)) \\
&= lcm(o(gh), o(xy)).
\end{align*}
Thus, $g_1 h_1 \in LCM(G \times H)$.

Now, suppose that $\gcd(\exp(G), \exp(H)) = 1$. To complete the proof, we need to show that
\[
LCM(G \times H) \subseteq LCM(G) \times LCM(H).
\]
Let $xt \in LCM(G \times H)$ where $x \in G$ and $t \in H$, and let $n$ be an integer. Then,
\[
(xt)^n = x^n t^n \in LCM(G \times H).
\]
Setting $x^n = g$ and $t^n = h$, we consider any $y \in G$. Then,
\[
o(ghy) \mid lcm(o(gh), o(y)).
\]
Since $gh = hg$, we obtain
\[
o(gh) \mid lcm(o(g), o(h)) = o(g) o(h).
\]
Furthermore,  
\[
(gh)^{o(g)} = g^{o(g)} h^{o(g)} = h^{o(g)}.
\]
Since $\gcd(o(g), o(h)) = 1$, it follows that $o(h^{o(g)}) = o(h)$, so $o(h) \mid o(gh)$.  
By a similar argument, $o(g) \mid o(gh)$. Therefore,  
\[
o(g) o(h) = lcm(o(g), o(h)) \mid o(gh).
\]

Since $\gcd(\exp(G), \exp(H)) = 1$, we conclude that $o(gh) = o(g) o(h)$. Also,  
\[
o(ghy) = o(gyh) = o(gy) o(h).
\]
Thus,
\begin{align*}
o(gy) o(h) &= o(ghy) \\
&\mid lcm(o(gh), o(y)) \\
&= lcm(o(g), o(y), o(h)) \\
&= lcm(o(g), o(y)) o(h).
\end{align*}
Consequently, $o(gy) \mid lcm(o(g), o(y))$, which implies that $x \in LCM(G)$.  
Similarly, by the same argument, $t \in LCM(H)$, so $xt \in LCM(G) \times LCM(H)$.  

Hence,  
\[
LCM(G \times H) \subseteq LCM(G) \times LCM(H).
\]
\end{proof}

The next example illustrates that $LC(G) \times LC(H)$ can be a proper subgroup of $LC(G \times H)$.
\begin{exam} Let $T=\langle x,y\rangle\times \langle a,b\rangle$ where 
$\langle x,y\rangle\cong\langle a,b\rangle\cong D_8$ and $o(x)=o(a)=4$ and
$o(y)=o(b)=2$.  We have $LC(\langle x,y\rangle)=\langle x\rangle$ and $LC(\langle a,b\rangle)=\langle a\rangle$. Therefore $LC(\langle x,y\rangle)\times LC(\langle a,b\rangle)=\langle x\rangle \times \langle a\rangle$. But
$xya\in LCM(T)\setminus \langle x\rangle \times \langle a\rangle$. 

\end{exam}
Let $G$ be a periodic group and $N$ a normal subgroup of $G$. The following example shows that $\frac{LC(G)N}{N}$ is not necessarily a subgroup of $LC(\frac{G}{N})$.
\begin{exam} Let $G=\langle x,y\rangle\times \langle a,b\rangle$ where 
$\langle x,y\rangle\cong\langle a,b\rangle\cong D_8$ and $o(x)=o(a)=4$ and
$o(y)=o(b)=2$.  Clearly, 
$xya\in LCM(G)$. Let $N=\langle a^2\rangle$.
Then $o(xyaN)=2$.
Since $$o(xya(y^{-1})N)=o(xN)=4\nmid lcm(o(xyaN),o(y^{-1}N))=2,$$ we have
$xyaN\not\in LC(\frac{G}{N}$). Hence $\frac{LC(G)N}{N}$ is not a subgroup of
$LC(\frac{G}{N})$.
\end{exam}

 Let $G$ be a finite group and let $p$ be a prime divisor of $|G|$.   For any subset $X$ of $G$, define  
\[
\Omega_{(p,1)}(X) = \{x \in X : x^p = 1\},
\]  
and let $\Omega_{p,1}(X)$ denote the subgroup of $G$ generated by $\Omega_{(p,1)}(X)$.
 
\begin{lem}\label{ccc}
Let $G$ be a $p$-group. Then $ \Omega_{p,1}(LCM(G))=\Omega_{(p,1)}(LCM(G))$.
\end{lem}
\begin{proof}
Let $z\in \Omega_{(p,1)}(LCM(G))\setminus \{1\}$.
Let $x_1,\ldots,x_r\in \Omega_{(p,1)}(LCM(G))$ such that
$z=x_1\ldots x_r$ where $r$ is minimal. We have  $o(x_i)=p$ for all $i=1,\ldots ,r$.
Since
\begin{align*}
o(z)&=o(x_1\ldots x_r)\\&\mid lcm(o(x_1),o(x_2\ldots x_r))\\&\vdots \\&\mid  lcm(o(x_1),\ldots ,o(x_r))=p.
\end{align*}
  Hence $exp(\Omega_{p,1}(LCM(G)))=p$, and so
$\Omega_{p,1}(LCM(G))=\Omega_{(p,1)}(LCM(G))$.
  
\end{proof}
\begin{rem}
    Let $G=C_3\wr C_3$ be the wreath product of $C_3$ by $C_3$.
    Then $G$ is a minimal irregular group and thus $|\mho_1(G)|=3$.
    Then $$\langle \{x\in LCM(G): o(x)=9\}\rangle=G,$$ so $LC(G)=G$.
    It follows that $\Omega_{3,1}(LC(G))\neq \Omega_{(3,1)}(LCM(G))$. 
\end{rem}

 If $G$ is an $LCM$-group, then $\exp(\Omega_{p,1}(G)) = p$ for any prime divisor $p$ of $|G|$, and $\frac{G}{\Omega_{p,1}(G)}$ is a $LCM$-group. The following lemma shows that the converse is also true.
  \begin{lem}\label{a2}
 Let $G$ be a finite. Then $\frac{G}{\Omega_{p,1}(G)}$ is an $LCM$-group and 
  $exp(\Omega_{p,1}(G))=p$ if and only if  $G$ is an $LCM$-group.
 \end{lem}
\begin{proof} 
Let $p$ be a prime divisor of $|G|$, and 
let $N:=\Omega_{p,1}(G)$.

($\Rightarrow$) If $G$ is an $LCM$-group, then 
$exp(N)=p$.
Also, for any $x\in G\setminus\{1\}$, we have $\langle x\rangle\cap N\neq 1$.
Then $o(xN)=o(x)/p$.
It follows that $exp(\Omega_n(\frac{G}{N}))=p^n$ for all divisor $p^n$ of $exp(\frac{G}{N}).$
Consequently, $\frac{G}{N}$ is an $LCM$-group.

($\Leftarrow$)
Since $G$ is an $LCM$-group, then $G$ is a nilpotent group.
Hence, we may assume that $G$ is a $p$-group.

Let $i\geq 1$ be an integer. Since  $\frac{G}{N}\in CP2$, from Theorem \ref{a},  $\Omega_{i}(\frac{G}{N})=\Omega_{(i)}(\frac{G}{N}).$ For the reason that $$N=\Omega_1(G)=\{z\in G: z^p=1\},$$ we have  $$\frac{\Omega_{i+1}(G)N}{N}\subseteq\Omega_{i}(\frac{G}{N})=\{xN: x\in  \Omega_{(i+1)}(G)\}.$$
Let $x\in \Omega_{i+1}(G)$.
Then $xN\in \Omega_{i}(\frac{G}{N})$, so  $xN=gN$ where $g\in \Omega_{(i+1)}(G)$.
Therefore $$x^{p^i}N=(xN)^{p^i}=(gN)^{p^i}=N.$$
It follows that $x^{p^{i+1}}=1$, since $exp(N)=p$, thus $x\in \Omega_{i+1}(G).$
Consequently,   $\Omega_{i+1}(G)=\Omega_{(i+1)}(G)$, and therefore Theorem \ref{a} provides that $G\in CP2$, and $G$ is a $LCM$-group.
\end{proof}
\begin{lem}\label{lcsub}
Let $G$ be a finite    group.
\begin{enumerate}
\item[(i)] Let $x,y\in LCM(G)$ of order prime number $p$. Then $xy\in LCM(G)$.
\item[(ii)] Let $\frac{H}{F}$ be a section of $G$. Suppose that
\[
\frac{NF}{F}\leq \frac{(LC(G)\cap H)F}{F}
\]
is a minimal normal subgroup of $\frac{H}{F}$, where $N\leq LC(G)\cap H$.
Then
\[
\frac{NF}{F}\subseteq LCM\left(\frac{H}{F}\right).
\]

\item[(iii)] If $G$ is an $LC$-nilpotent group, then every section  $\frac{H}{F}$  of $G$ is also an
$LC$-nilpotent group.

\end{enumerate}
\end{lem}

\begin{proof}

\smallskip
\noindent\textbf{(i)}
Since \[o(xy)\mid lcm(o(x),o(y))=p,\]
we have $o(xy)\mid p$.
Let $z\in G\setminus \{1\}$ be an $p$-element.
Then 
\[o((xy)z)=o(x(yz)\mid lcm(o(x),o(yz))\mid lcm(o(x),o(y),o(z))=o(z)=lcm(o(xy),o(z)).\]

Let $n\ge 1$ be an integer.
Then by induction we can prove that
\begin{align*}
    o((xy)^nz)&=o(x(y(xy)^{n-1}z))\\&\mid  lcm(o(x),o(y(xy)^{n-1}z))\\&\mid lcm(o(x),o(y),o((xy)^{n-1}z))\\&\mid o(z)=lcm(o((xy)^n),o(z)).
\end{align*}
 \[\]

Then $xy\in LCM(G).$

\smallskip
\noindent\textbf{(ii)}
We argue by induction on $\left|\frac{H}{F}\right|$.
Since $LC(G)$ is a nilpotent   group, $\frac{NF}{F}$  is abelian, so
\(
\left|\frac{NF}{F}\right|=p^m
\)
for some prime $p$. Suppose, by contradiction, that there exist
$u\in N$ and a $p$-element $v\in H$ such that
\[
o(uvF)\nmid lcm(o(uF),o(vF)).
\]

Set
\(
\frac{K}{F}=\frac{N\langle v\rangle F}{F}.
\)
If $K=H$, then $\frac{H}{F}$ is a $p$-group. Since
\(
\frac{NF}{F}\leq \frac{H}{F}
\)
is a minimal normal subgroup of \(\frac{H}{F}\), we must have 
\(
\frac{NF}{F}\leq Z\left(\frac{H}{F}\right).
\)
Consequently, $\frac{H}{F}$ is abelian, and hence
\[
\frac{H}{F}=LCM\left(\frac{H}{F}\right),
\]
which contradicts the choice of $u$ and $v$.
Therefore,
\(
\frac{K}{F}\neq \frac{H}{F}.
\)
Let $A(N)$ denote the set of all minimal normal subgroups
\(
\frac{UF}{F}
\)
of $\frac{K}{F}$ such that $U\leq N$. For every
\(
\frac{UF}{F}\in A(N),
\)
we have
\(
U\leq LC(G)\cap H .
\)
Hence, by the induction hypothesis,
\[
\frac{UF}{F}\subseteq  LCM\left(\frac{K}{F}\right).
\]

Moreover, since 
\[
\frac{NF}{F}
=
\left\langle
\left\{
\frac{UF}{F}:\frac{UF}{F}\in A(N)
\right\}
\right\rangle .
\]
 
By Case (i),
\[
\frac{NF}{F}\subseteq  LCM\left(\frac{K}{F}\right).
\]

It follows that
\[
o(uvF)\mid lcm(o(uF),o(vF)),
\]
which is a contradiction. Hence,
\[
\frac{NF}{F}\subseteq LCM\left(\frac{H}{F}\right).
\]

\smallskip
\noindent\textbf{(iii)}
We prove the result by  induction on \(|G|\).

The base case $|G|=1$ is trivial. So suppose that $|G|>1$. Let  $S:=\frac{H}{F}$ be a section of $G$.
Let $i\geq 1$ be such that
\[
\frac{H_i}{F}:=LC_i(S)=LC_{i+1}(S).
\]
Let $R=\frac{(LC(G)\cap H)H_i}{H_i}$.
If $R\neq 1$, then there exists a minimal normal subgroup $\frac{NH_i}{H_i}$  such that $N\leq (LC(G)\cap H)$.
By part (ii) \[\frac{NH_i}{H_i}\subseteq LCM\left(\frac{H}{H_i}\right).\]
Then \[LC\left(\frac{H}{H_i}\right)\cong LC\left(\frac{H/F}{H_i/F}\right)=\frac{LC_{i+1}(S)}{LC_i(S)}\neq 1,\]
  which is a contradiction. Thus $R=1$, and so 
\[
\frac{(LC(G)\cap H)F}{F}\leq \frac{H_i}{F}=LC_i(S).
\]

Since $G$ is $LC$-nilpotent, the quotient
\(
\frac{G}{LC(G)}
\)
is again an $LC$-nilpotent group. Since \(|\frac{G}{LC(G)}|<|G|\), by the induction hypothesis,
\[
\frac{LC(G)H}{LC(G)}
\cong
\frac{H}{H\cap LC(G)}
\]
is an $LC$-nilpotent group.
Because
\[
\left|\frac{H}{H\cap LC(G)}\right|<|G|,
\]
another application of the induction hypothesis gives
\[
\frac{
H/(H\cap LC(G))
}{
H_i/(H\cap LC(G))
}
\cong
\frac{H}{H_i}\cong \frac{H/F}{H_i/F}=\frac{S}{LC_i(S)}
\]
is an $LC$-nilpotent group.

Therefore, by the definition of the $LC$-series,
$S$ is an $LC$-nilpotent group.

\end{proof}

\begin{lem}\label{semi2}
     Let $G$ be a finite group and let $G=P\rtimes K$ where $gcd(|P|,|K|)=1$.
Then $G$ is an $LC$-nilpotent group if and only if $P$ and $K$ are $LC$-nilpotent groups. 
 \end{lem}
 \begin{proof}

 ($\Leftarrow$) We proceed by induction on $|G|$.
 If $|G|=|P|$, then $G=P$ is an $LC$-nilpotent group. So suppose $|G|>|P|$.  Let $p$ be a prime divisor of $|LC(P)|$.
     By Proposition \ref{equ}, $LC_p(P)\leq LC(G)$. Then $LC(G)\neq 1$. Since $gcd(|P|,|K|)=1$ we have 
     \[\frac{G}{LC(G)}\cong \frac{P}{LC(G)\cap P}\rtimes\frac{K}{LC(G)\cap K}.\]

     By Lemma \ref{lcsub}  (iii), $\frac{P}{LC(G)\cap P}$ and $\frac{K}{LC(G)\cap K}$ are $LC$-nilpotent groups. By the induction hypothesis, $\frac{G}{LC(G)}$ is an $LC$-nilpotent group, thus $G$ is an $LC$-nilpotent group. 

      ($\Rightarrow$)
      Then $G$ is an $LC$-nilpotent group. By Lemma \ref{lcsub},  $P$ and $K$ are $LC$-nilpotent  groups.

 \end{proof}
 \begin{lem}\label{semii}
    Let $G$ be a finite group and let $x \in LCM(G)$. Suppose $N$ is a normal subgroup of $G$ such that $\gcd(|N|, o(x)) = 1$. Then,  
    \[
    xN \in LCM\left(\frac{G}{N}\right).
    \]
\end{lem}
 \begin{proof}
     By Proposition \ref{equ}, we may assume that $x$ is a $p$-element for some prime divisor of $|G|$.
     Since $p\nmid |N|$, we have $\langle y\rangle \cap N=1$, so $o(yN)=o(y)$ for any $p$-element $y$ of $G$. It follows that 
     \begin{align*}
       o(xyN)=o(xy)&\mid lcm(o(x),o(y))\\&=lcm(o(xN),o(yN)).
     \end{align*} 
     Consequently, $xN\in LCM_p(\frac{G}{N})\subseteq LCM(\frac{G}{N})$
 \end{proof}
 
\begin{rem}
    
 Let $G:=SmallGroup(243,4)$. Then $G$ is a $LCM$-group. Also, there exists a minimal normal subgroup $N$ of $G$ such that $\frac{G}{N}=\Omega_{3,1}(\frac{G}{N})$ and $exp(\frac{G}{N})=9$. Hence  there exists $x\in \Omega_{3,1}(LCM(G))$ such that $xN\not\in LCM(\frac{G}{N}).$
  \end{rem}

 
 \section{$LC$-nilpotent groups}
 In this section, we establish the sufficient and necessary conditions for a finite group $G$ to be an $LC$-nilpotent group. We start with the following definition.
   
 \begin{defn}
    Let $G$ be a finite group. We call $G$ a $LC$-simple group whenever $LC(G)=1$.
\end{defn}
\noindent The following three Lemmas help us to prove Theorem \ref{lc5}.
\begin{lem}\label{abel}
    Let $G=ND$ where $N$  is an elementary abelian $p$-group and $D$ is an elementary abelian $q$-group.
    Suppose for any characteristic subgroup of $E$ of $N$ and any subgroup $F$ of $D$, $EF$
    is not a Frobenius group.
 Then $q\mid |\operatorname{Fit}(G)|$.

\end{lem}
\begin{proof}
    We proceed by induction on $|G|$.
    First suppose $q^2\mid |D|$.
Let $M$ be a maximal normal subgroup of $G$ such that $N\leq M$.
If $q\mid |M|$, then by the induction  $q\mid |\operatorname{Fit}(M)|$, so $q\mid |\operatorname{Fit}(G)|$.
So $q^2\nmid |D|$, and so $N$ is a maximal subgroup of $G$ and $D$ is a cyclic group.   By the Fitting theorem
    $N=C_N(D)\times [N,D]$. 
    Since $[N,D]=G'$ is a characteristic subgroup of $G$, by our assumption  $G'D$ is not a Frobenius group.
    By the induction hypothesis, $q\mid |\operatorname{Fit}(G'D)|$, so
$D\leq \operatorname{Fit}(G'D)$, thus  $C_{G'}(D)=G'\neq 1$, which leads to a contradiction.
    
\end{proof}

Let $\mathfrak{F}$ be the set of all $2$-Frobenius groups $G$ of type $(p, q, p)$. The Lemma \ref{2fr} establishes that for any $G \in \mathfrak{F}$, we have $LC(G) = 1$. Hereafter, when we write $G = ABC \in \mathfrak{F}$, it signifies that $A$ and $AB$ are normal subgroups of $G$, and $AB$ and $BC$ are Frobenius groups with kernels $A$ and $B$, and complements $B$ and $C$, respectively.

The following auxiliary Lemma is used to prove Lemma \ref{2fr}.

\begin{lem}\label{Rulin}(Lemma 1.7 \cite{13}) Let $G = ABC$ be a $2$-Frobenius group. Suppose that $AC$ is a $p$-group. Then $\exp(AC) \geq p^2$.
    
\end{lem}

\begin{lem}\label{2fr}
    Let $G=ABC\in \mathfrak{F}$. Then $LC(G)=1$.
\end{lem}
\begin{proof}
    Suppose for a contradiction that $LC(G)\neq 1$.
    Then $LC(G)\leq A=\operatorname{Fit}(G)$.  Let $P\in \operatorname{Syl}_p(G)$.
    We have $P=A\rtimes C$.
    Let $N=\Omega_{3,1}(LCM(G))$, and let $z\in C$ of order $p$.
Then $H=NB\langle z\rangle$ is a $2$-Frobenius group. As $N\subseteq LCM(H))$, we deduce that    $exp(N\langle z\rangle)=p$.
By Lemma \ref{Rulin}, $exp(N\langle z\rangle)\geq p^2$, which is a contradiction.
\end{proof}
 The following Lemma is a  special case in  Theorem 
\ref{lc5}. 
 
 \begin{lem}\label{3211}
    Let $G$ be a finite solvable group such that $\mathfrak{F}$ does not contain any section of $G$, and let $N$ be a minimal normal subgroup of $G$. Furthermore, suppose that for any solvable group $H$ for which $\mathfrak{F}$ does not contain any section of $H$ and $|H| < |G|$, the group $H$ is a $LC$-nilpotent group. Then $N \subseteq LCM(G)$.
\end{lem}
\begin{proof}
Since $G$ is a solvable group, $|N|=p^m$ for some prime number $p$ and integer $m\geq 1$.
 We proceed by induction on $|G|$.

For $G = N$, the result is trivial. Assume $G\neq N$. 
Let $x\in N\setminus\{1\}$, and  
let $y\in G\setminus N$ be a $p$-element in $G$ such that 
$o(xy)\nmid lcm(o(x),o(y))$.

First suppose that $G$ has another minimal normal subgroup $D$. 
By our assumption $\mathfrak{F}$ does not contain any section of $\frac{G}{D}$, then by the induction hypothesis,  
     $\frac{ND}{D}\subseteq LCM(\frac{G}{D})$. We have $o(x)=p$

    If $\langle xy\rangle \cap N=1$, we have 
    \begin{align*}
   o(xy)&=o(xyN)\\&=o(yN)\\&\mid o(y)\\&=lcm(o(x),o(y)),
    \end{align*}
  which is a contradiction.
    Thus $\langle xy\rangle \cap N\neq 1$, $\langle xy\rangle \cap D=1$ and then
    $o(xyD)=o(xy)$. It follows from $\frac{ND}{D}\subseteq LCM(\frac{G}{D})$ that 
    \begin{align*}
   o(xy)&=o(xyD)\\&\mid lcm(o(xD),o(yD))\\&\mid lcm(o(x),o(y)),
    \end{align*}   which leads to a contradiction.
  
So  $N$ is the unique normal minimal subgroup of $G$.
 Then $G=N\rtimes M$ where $M$ is a maximal subgroup of $G$ and $\operatorname{Fit}(G)$ is a $p$-group.  If $\varPhi(\operatorname{Fit}(G))\neq 1$, then $$N\leq \varPhi(\operatorname{Fit}(G))\leq \varPhi(G)\leq M,$$ which is a contradiction, as $M\cap N=1$.
 Thus $\varPhi(\operatorname{Fit}(G))=1$, and so $\operatorname{Fit}(G)$ is an elementary abelian $p$-subgroup of $G$.
 Since $N$ is the unique minimal normal subgroup of $G$, we conclude that $N=\operatorname{Fit}(G)$.

  Since  $\mathfrak{F}$   does not contain any  section of  $M$, by our assumption, $M$ is an $LC$-nilpotent group.
  So $LCM(M)\neq 1$.
Let $D\subseteq LCM(M)$ be a normal minimal subgroup of $M$.  

First suppose that $H_c:=ND\langle c\rangle\neq G$ for any   $p$-element $c$ of $G$.
Let $U\leq N$ be a minimal normal subgroup of $ND$.
Let $g\in G$. We claim that $U^g$ is a normal minimal  subgroup of $ND$.
Let $h\in ND$. Since $ND\lhd G$, there exists $h'\in ND$ such that $gh=h'g$.
It follows that
\[U^{gh}=U^{h'g}=U^{g},\]
 so $U^g\lhd ND$.
Let $R\leq U^g$ be a normal minimal subgroup of $ND$.
Then  $R^{g^{-1}}\lhd ND$.
Since $U$ is a normal  minimal  subgroup of $ND$, we must have $R^{g^{-1}}=U$, consequently, $R=U^g$, is a normal minimal subgroup of $ND$, as claimed.

Let $$\Delta=\{U^g:g\in G\}=\{U^{g_1},\ldots,U^{g_k}\}.$$
Let $S_j=U^{g_1}U^{g_2}\ldots U^{g_j}$ for $j=1,...,k$.
Let $r\geq 1$  be the smallest integer such that 
$S_r^{g}=S_r:=S$ for any $g\in \{g_1,...,g_k\}$.
Since $S^{g}=S$ for all $g\in \{g_1,...,g_k\}$, we conclude that $S\lhd G$.

Let $1\le i\le r$.
If there exists  $h\in U^{g_i}\cap \Pi_{j\neq i}U^{g_j}$ such that $h\neq 1$, then 
\[U^{g_i}=\langle h^{ND}\rangle\subseteq U^{g_i}\cap \Pi_{j\neq i,j\le r}U^{g_j},\]
which is a contradiction by minimality of $r$.
Hence, 
\(U^{g_i}\cap \Pi_{j\neq i}U^{g_j}=1\) for all $1\le i\le r$, and so
$$S=U^{g_1}\times  \ldots \times U^{g_r}$$ is a product of minimal normal subgroups of $ND$.
Since $S\lhd G$, $S\le N$ and $N$ is a minimal normal subgroup of $G$, we must have $N=S$.

Any $U^{g_i}$ is contained in some minimal normal subgroup $F_i$ of $H_c$, as $U^{g_i}$ is a minimal normal subgroup of $ND$ and $ND\leq H_c$.
By the induction hypothesis, 
$$U^{g_i}\leq F_i\leq \Omega_{(p,1)}(LCM(H_c)),$$ for all $i$, and so $N\leq \Omega_{(p,1)}(LCM(H_c))$.
Hence,  for any $x\in N$, we have 
$o(x^nc)\mid lcm(o(x^n),o(c))$
for any integer $n$ and any $p$-element $c$. 
In particular, 
$o(xy)\mid lcm(o(x),o(y))$, 
which is a contradiction.

Therefore $H_c=G$ for some $p$-element $c\in G$. We may assume that $c=y$. Let  $H:=H_y=G$.
 If $G$ is a $p$-group, then
$N\leq Z(G)$, so $o(xy)\mid lcm(o(x),o(y))$, 
which is a contradiction.  
So assume $G$ is not a $p$-group. Next   $\pi(|G|)=\{p,q\}$ where $\pi(|D|)=\{q\}\neq \{p\}$, as $G=H$.  
 Since  $H\not\in \mathfrak{F}$,   $H$ has an element $w$ of order $pq$.
 Since $y$ is an $p$-element and $\frac{G}{ND}=\langle yND\rangle$,  all elements of order $pq$ are in $ND\langle w\rangle$, we    consider the following two cases:

 {\bf Case 1.}
 First suppose 
$M$ has an element of order $pq$. We may assume that $w\in M$. Let  $R:=D\langle w\rangle$, and let   $v=w^p$.
Since $p\mid |C_R(v)|$ and $D\leq C_R(v)$, we conclude that $C_R(v)=R$. 

As  $\frac{M}{D}$ is cyclic, we conclude that $R\lhd M$.
Since $Z(R)$ is a characteristic subgroup of $R$, we have $Z(R)\lhd M$.
Since $Z(R)\lhd M$, and $D$ is a minimal normal subgroup of $M$, we have $D=\langle v^M\rangle\le Z(R)$.
Therefore $R$ is an abelian group.
Then $ \langle w^q\rangle\lhd M$.
It follows that 
$N\langle w^q\rangle\lhd G$, and so 
$\langle w^q\rangle\leq Fit(G)$, which is a contradiction.

 {\bf Case 2.} 
 Suppose $M$ does not have any element of order $pq$, and let $b\in M$ be a $p$-element such that $o(bDN)=o(yDN)$.
 Hence
 $w\in ND$.
Let $E\leq N$ be a characteristic subgroup of $ND$. Since $ND\lhd G$, we conclude that $E\lhd G$. 
Then $ED\langle y\rangle$ is a subgroup of $G$.
Since $\mathfrak{F}$ does not contain any  section of $G$, and $D\langle b\rangle$ is a Frobenius group, $ED$ is not a Frobenius group.
By Lemma \ref{abel}, $q\mid |\operatorname{Fit}(ND)|$, therefore $q\mid |C_{ND}(N)|$.
Since $N=\operatorname{Fit}(G)$ and $G$ is a solvable group, we have $C_G(N)=N$, which is our final contradiction. 
Thus $N  \subseteq LCM(G)$. 

 \end{proof}
 \begin{thm}\label{lc5}
    Let $G$ be a finite solvable group. Then $G$ is an $LC$-nilpotent group if and only if $\mathfrak{F}$ does not contain any section of $G$.
 \end{thm}
 \begin{proof}
 Let $G$ be a minimal counter example for this theorem, i.e, $G$ is a minimal counterexample in the following way: either $G$ is an $LC$-nilpotent group  and  
 $\mathfrak{F}$   contains a  section of $G$, or
 $\mathfrak{F}$ does not contain any section of  $G$  and $G$ is not an $LC$-nilpotent group. 
 Then for any group $H$ with $|H|<|G|$  we have 
 $H$ is an $LC$-nilpotent group if and only if 
 $\mathfrak{F}$ does not contain any section of $H$.

First suppose that $G$ is an $LC$-nilpotent group. We show that $\mathfrak{F}$ does not contain any section of $G$.
    Let $\frac{K}{N}$ be a section of $G$ such that $\frac{K}{N}$ is a $2$-Frobenius group of type $(p, q, p)$. If $K\neq G$ or $N\neq 1$ by the induction hypothesis, $\frac{K}{N}$ is not a $2$-Frobenius group, which is a contradiction.   By Lemma \ref{lcsub}, $\frac{K}{N}$ is an $LC$-nilpotent group.
    Therefore  we may assume that $K=G$ and $N=1$. 
    By Lemma \ref{2fr},
    $LC(G)=1$, which leads to a contradiction.

    Now, suppose any section of $G$ does not contain any $2$-Frobenius subgroup of type $(p,q,p)$. We show that $G$ is an $LC$-nilpotent group. 
     Let $N$ be a minimal normal subgroup of $G$ of order $p^m$, where $p$ is a prime. 
By Lemma \ref{3211}, $N \subseteq LCM(G)$.

Hence, $LC(G)\neq 1$.
By minimality of $G$, 
$\frac{G}{LC(G)}$ is an $LC$-nilpotent group and so $G$ is an $LC$-nilpotent group.
\end{proof}
\begin{rem}
     Note that Theorem \ref{lc5} is not valid if we drop the condition of solvability, as $G=A_5$ does not contain any $2$-Frobenius subgroup of type $(p,q,p)$, while $LC(A_5)=1$.
 \end{rem}
 
 From Theorem \ref{lc5} it immediately follows that the algebraic system of all LC-nilpotent group is a variety.
 \begin{cor}\label{ques}
 Let $\mathfrak{L}$ be the set of all $LC$-nilpotent finite groups.
     Then  subgroups and  quotients of elements of $\mathfrak{L}$ are in $\mathfrak{L}$. Also, for  $G,H\in \mathfrak{L}$, we have $G\times H\in \mathfrak{L}$.
 \end{cor}

Let $G$ be a finite group such that for any proper subgroup $H$ of $G$, $LC(H)\neq1$. Is it true that $G$ is a solvable group? The answer is no.
For example, let $A_5$ be the alternating group of degree five. Then  
    $LC(A_5)=1$, but 
    $LC(H)\neq 1$ for any proper subgroup $H$ of $A_5$.  
In the following theorem we find the structure of  finite solvable groups $G$, such that $LC(G)=1$, but $LC(H)\neq 1$ for any proper subgroup $H$ of $G$.
 
\begin{thm}
    Let $G$ be a finite solvable group such that $LC(G)=1$. Then for any proper subgroup $H$ of $G$, $LC(H)\neq 1$ if and only if $G=ABC$ is a $2$-Frobenius group of type $(p,q,p)$ where $A$ is a  minimal  normal subgroup, $|B|=q$ and $|C|=p$.
\end{thm}
\begin{proof}
($\Leftarrow$) It follows from Lemma \ref{2fr}. 

 \noindent ($\Rightarrow$)   From Theorem \ref{lc5}, $G$ has a $2$-Frobenius subgroup $K=ABC$ of type $(p,q,p)$.  From Lemma \ref{2fr} it follows that 
    $LC(K)=1$. Then $K=G$.
Let $N$ be a  minimal normal subgroup of $G$.
Then $NBC$ is a $2$-Frobenius subgroup of $G$  of type $(p,q,p)$. Since $LC(NBC)=1$, we have $NBC=G$.
Let $D\leq B$ of order $q$.
Then $NDC$ is a $2$-Frobenius subgroup of $G$ of type $(p,q,p)$. Since $LC(NDC)=1$, we have $NDC=G$. By the same argument $|C|=p$.
\end{proof}

Next we prove that any
finite supersolvable group is an $LC$-nilpotent group.
 \begin{lem}\label{super}
   Let   $G$ be a supersolvable group. Then $G$ does not contain any $2$-Frobenius subgroups of type $(p,q,p)$. 
\end{lem}
\begin{proof}
    We proceed by induction on $|G|$.
    Let $N$ be a normal subgroup of $G$ of order $p$ where $p$ is a prime number.
    By the induction hypothesis, $\frac{G}{N}$ does not contain any $2$-Frobenius subgroup.
    Suppose that $H$ is a $2$-Frobenius subgroup of $G$, then
     $\operatorname{Fit}(H)$ is a non-cyclic  subgroup of $H$.
    Also, $H=BC$ where $B=\operatorname{Fit}(H)\rtimes L$ is a Frobenius group and both $C$ and $\frac{B}{\operatorname{Fit}(H)}$ are cyclic groups.
    Clearly, $N\leq H$.
   Since $|\operatorname{Aut}(N)|=p-1$, we have $[H:C_H(N)]\mid p-1$.
   It follows that $q:=|L|\mid p-1$.
   Then $q\leq p-1$.
   Since $\frac{H}{\operatorname{Fit}(H)}$ is a Frobenius group, we have $p\mid |C|\mid q-1$, which is a contradiction.    
\end{proof}
\begin{cor}\label{5}
Let $G$ be a supersolvable  group. Then
$G$ is a finite $LC$-nilpotent group.
\end{cor}
 \begin{proof}
 Since subgroups and quotient groups of supersolvable groups are supersolvable, the result 
 follows from Lemma \ref{super} and Theorem \ref{lc5}.
 \end{proof}

 Let $A_4$ be the alternating group of degree four. Then $A_4$ is not a supersolvable group, but
$LC_1(A_4)$ is a group of order $4$ and $LC_2(A_4)=A_4$.
Hence $A_4$ is a $LC$-nilpotent group, which is not a supersolvable group. 
Therefore the set of all $LC$-nilpotent groups is larger than the set of all finite supersolvable groups. 
  \begin{lem}\label{abb}
      Let $G$ be a finite group such that 
    for any $x,y\in G\setminus\{1\}$, we have $LC(\langle x,y\rangle)\neq 1$.
    Let $N$ be a normal minimal abelian subgroup of $G$ of order $p^m$.
    Then $N\leq \Omega_{p,1}(LCM(G))$.
  \end{lem}
  \begin{proof}
    We proceed by induction on $|G|$.
First, suppose that for any $p$-element $y \in G$, the subgroup $H_y := N\langle y \rangle \neq G$. We have  $N = U_1 \times \ldots \times U_k$, where $U_1, \ldots, U_k$ are minimal normal subgroups of $H_y$. By the induction hypothesis, 
    \[
    N = U_1 \times \ldots \times U_k \leq \Omega_{p,1}(LCM(H_y)),
    \]
    thus $o(ay) \mid lcm(o(a), o(y))$ for any $a \in N$ for any $p$-element $y\in G$. Hence, $N \leq \Omega_{p,1}(LCM_p(G))$, and therefore $N \leq \Omega_{p,1}(LCM(G))$.

    Now, suppose $H_y = G$ for some $p$-element $y$ of $G$. Then, $N \leq Z(G) \leq LCM(G)$.
    This completes the proof.
\end{proof}
\noindent Next we are ready to prove the main result of our paper.
\begin{thm}
    Let $G$ be a finite  group. Then   the following conditions are equivalent:

\begin{enumerate}
    \item   $G$ is an $LC$-nilpotent group.
    
    \item   No section  of $G$ contains any $2$-Frobenius subgroup of type $(p, q, p)$.

  \item   For any $x,y\in G\setminus\{1\}$, we have $LC(\langle x,y\rangle)\neq 1$. \end{enumerate}

\label{mainthe1}
\end{thm}
\begin{proof}
    The equivalence $(1)\Leftrightarrow(2)$ follows directly from Theorem \ref{lc5}.

\noindent $(1)\Rightarrow(3):$      
    Suppose for any $x, y \in G$ the subgroup $\langle x, y \rangle$ is an $LC$-nilpotent group, then it follows that $LC(\langle x, y \rangle) \neq 1$, for any $x, y \in G$.
    
\noindent $(1)\Leftarrow(3):$
Let $G$ be the minimal counterexample.
Then for any group finite group $H$ of order less than $G$, if $H$ satisfies  one the conditions of the theorem it satisfies all other conditions. 

Suppose first that \(G\) is not a solvable group. By the induction hypothesis,
every proper subgroup of \(G\) is \(LC\)-nilpotent. Hence \(G\) is a minimal
non-solvable group and therefore \(G\) is a minimal simple group.

It is well known that every minimal simple group is generated by two elements.
Let \(x,y\in G\) be such that \(G=\langle x,y\rangle\). Since \(G\) is not
\(LC\)-nilpotent, we have
\[
1\neq LC(\langle x,y\rangle)=LC(G)\leq \operatorname{Fit}(G)=1,
\]
which is a contradiction. Therefore, \(G\) must be solvable.

Let \(N\) be a minimal normal subgroup of \(G\). By Lemma~\ref{abb}, we have
\(
N\leq LC(G).
\)

Since $G$ is not $LC$-nilpotent group  \(G\neq N\).
We consider the following two cases.

\noindent
\textbf{Case 1.}
Assume that there exist \(x,y \in G \setminus N\) such that
\[
LC\!\left(\frac{H}{N}\right)=1,
\quad \text{where } \frac{H}{N}=\langle xN,yN\rangle .
\]
Then \(\frac{H}{N}\) is not an \(LC\)-nilpotent group.

If \(H \neq G\), then by the induction hypothesis, \(H\) is \(LC\)-nilpotent. Hence, by Theorem~\ref{lc5}, \(H\) has no section belonging to \(\mathfrak L\), and therefore \(\frac{H}{N}\) also has no section in \(\mathfrak L\). Again by Theorem~\ref{lc5}, this implies that \(\frac{H}{N}\) is \(LC\)-nilpotent, a contradiction. Thus \(H=G\), and so \(LC(\frac{G}{N})=1\).

First suppose that \(\langle x,y\rangle \neq G\).
Set $R=N \cap \langle x,y\rangle.$
Since $N\lhd G$, we have $R\lhd \langle x,y\rangle$.
Since $R\le N$ and $N$ is abelian, $N\le N_G(R)$. So $R\lhd G$.
Since $N$ a minimal normal subgroup of \(G\) and $N\nleq \langle x,y\rangle$, we have
\[N \cap \langle x,y\rangle=1.\]

By the minimality argument, \(\langle x,y\rangle\) is \(LC\)-nilpotent. Consequently,
\[
\frac{G}{N} \cong \langle x,y\rangle,
\]
is \(LC\)-nilpotent, a contradiction.

Therefore, we must have \(\langle x,y\rangle = G\).

Let \(D\) be another minimal normal subgroup of \(G\).
By Lemma \ref{abb},
$D\subseteq LCM(G)$.
If \(\gcd(|D|,|N|)=1\), then by Lemma~\ref{semii},
\[
\frac{DN}{N} \leq LC\!\left(\frac{G}{N}\right)=1,
\]
which is a contradiction. Hence \(DN\) is a \(p\)-group.

Let \(u \in D \setminus \{1\}\). Since \(LC(\frac{G}{N})=1\), there exists   \(p\)-element \(v\) of \(G\) such that
\[
o(uvN) \nmid lcm(o(uN),o(vN)).
\]

First suppose that \(ND\langle v\rangle \neq G\). Let \(H\) be a maximal subgroup of \(G\) such that
\[
ND\langle v\rangle \leq H.
\]
By the minimality of \(G\), the subgroup \(H\) is \(LC\)-nilpotent, and hence so is \(\frac{H}{N}\).
Since \(|\frac{H}{N}|<|G|\), the conditions (1) and (3) of the theorem are equivalent for \(\frac{H}{N}\). It follows that   \(LC(\langle wN,gN\rangle)\neq 1\) for all nontrivial \(wN,gN \in \frac{H}{N}\).

Let \(\frac{U}{N} \leq \frac{DN}{N}\) be a minimal normal subgroup of \(\frac{H}{N}\). As in the proof of Lemma~\ref{3211}, each conjugate \(\frac{U^g}{N}\) is a minimal normal subgroup of \(\frac{H}{N}\).
Let \(\{\frac{U^{g_1}}{N},\ldots,\frac{U^{g_k}}{N}\}\) be the set of all distinct conjugates of \(\frac{U}{N}\). Then, as in Lemma~\ref{3211}, there exists \(1 \le r \le k\) such that
\[
\frac{DN}{N}=\frac{U^{g_1}}{N}\times \cdots \times \frac{U^{g_r}}{N}.
\]
By Lemma~\ref{abb}, each \(\frac{U^{g_i}}{N}\leq \Omega_{p,1}\!\big(LCM(\tfrac{H}{N})\big)\), hence
\[
\frac{DN}{N} \leq \Omega_{p,1}\!\big(LCM(\tfrac{H}{N})\big).
\]
It follows that
\[
o(uvN)\mid lcm(o(uN),o(vN)),
\]
a contradiction.

Therefore, \(G = ND\langle v\rangle\), and thus \(G\) is a \(p\)-group. Consequently, \(G\) is \(LC\)-nilpotent, which is a contradiction.

Hence, \(N\) is the unique minimal normal subgroup of \(G\). Therefore,
\[
G = N \rtimes M
\]
for some maximal subgroup \(M\) of \(G\). By the minimality of \(G\), \(M\) is \(LC\)-nilpotent, so
\[
\frac{G}{N} \cong M
\]
is also \(LC\)-nilpotent, contradicting the assumption.

\noindent 
{\bf Case 2.}
Thus, for all \(x,y\in G\),
\[
LC(\langle xN,yN\rangle)\neq 1.
\]
By the minimality of $G$, it follows that
\(
\frac{G}{N}
\)
is an \(LC\)-nilpotent group.

By Lemma~\ref{semii},
\(
N\leq LC(G).
\)
Furthermore, by Theorem~\ref{lc5}, the group \(\frac{G}{N}\) has no section belonging
to \(\mathfrak L\). Hence,
\[
\frac{G/N}{LC(G)/N}\cong \frac{G}{LC(G)}
\]
also has no section belonging to \(\mathfrak L\).

Applying Theorem~\ref{lc5} once more, we conclude that
\(
\frac{G}{LC(G)}
\)
is an \(LC\)-nilpotent group, so $G$ is an $LC$-nilpotent group, which is our final contradiction.
\end{proof}

The Theorem \ref{mainthe1} gives a more general answer to the Question \ref{question1}. 

 \section{Maximal class}
In this section we classify the $LC$-nilpotent groups of maximal $LC$-class and $LC$-nilpotent groups $G$ such that  $\frac{LC_{i+1}(G)}{LC_i(G)}$ is cyclic for any positive integer $i$.
 
 \begin{thm}
    Let $G$ be a finite non-abelian group. Then $G$ is an $LC$-nilpotent group of maximal $LC$-class if and only if $G$ is a non-abelian group of order $pq$ where $p < q$ are prime numbers.
   \end{thm}
    \begin{proof}
($\Leftarrow$) If $G$ is a non-abelian group of order $pq$, then  $G$ is an $LC$-nilpotent group of maximal $LC$-class.

\noindent ($\Rightarrow$) We proceed by induction on $k$ where $LC_k(G)=G$ and $LC_{k-1}(G)\neq G$. Let $N := LC_1(G)$, and let $|N| = q$ where $q$ is a prime number. Since $G$ is not abelian, $k>1$.

First, suppose that $k = 2$.  Then $|G| = pq$, where $p$ is a prime number. Since $G$ is not abelian, $p < q$. Suppose that $k > 2$. By the induction hypothesis,  $\frac{G}{N}$ is a non-abelian group of order $ps$, where $p \neq s$ are prime numbers.
As $\frac{G}{N}$ is a non-abelian group of order $ps$,   $|G| = pqs$.    We may assume that $p<s$.
If $\operatorname{Fit}(G)\neq N$, then $\operatorname{Fit}(G)$ is an abelian group, as $|\operatorname{Fit}(G)|$ is a product of two prime numbers.
Since $p<s$, we have $|\operatorname{Fit}(G)|=qs$.
If $s=q$ or $s\neq q$, then by Corollary \ref{42}, $\operatorname{Fit}(G)\leq N$, which leads to a contradiction.
Thus $\operatorname{Fit}(G)=N$, and $|\operatorname{Fit}(G)|=q$.
Then $\operatorname{Aut}(N)\cong C_{q-1}$ is a cyclic group and so $\frac{G}{C_G(N)}=\frac{G}{N}$ is an abelian group. Hence, $k=2$, which is a contradiction. 
    \end{proof}
 \begin{lem}
     Let $G\ncong D_8$ be a $2-$group of maximal class, and let $N$ be a normal non-cyclic subgroups of $G$.
     Then $N$ is of maximal class.
 \end{lem}
 \begin{proof}
 We proceed by induction on $|G|$
    Let $F:=\langle f\rangle$ be the cyclic maximal subgroup of $G$.
    Let $D\leq F$ be a subgroup of order two.
    Then 
    $|F\cap N|=|F|/2$. 
    Hence, $N$ has a cyclic subgroup of index  two.
    If $|N|=4$, then 
    $N\cong C_2\times C_2$.
    Since $G\cong D_8$, $G$ does not contain any normal subgroups of type $C_2\times C_2$. Then $|N|>4$.
    If $\frac{G}{D}\ncong D_8$, by induction hypothesis, 
    $\frac{N}{D}$ is of maximal class, therefore
    $N$ is of maximal class.
    If $\frac{G}{D}\cong D_8$, then $|G|=16$, and the proof is complete.
 \end{proof}
 We need the following technical Lemma to prove Theorem \ref{max}.
 \begin{lem}\label{sub4}
     Let $G=N\langle z\rangle$ be a finite group $2$-group where $N$ is a normal cyclic subgroup $N$ of index 4 and $o(zN)=4$.
     If $|\Omega_1(G)|=4$, then $G$ has an element of $v$ of order $4$ such that
     $\langle v\rangle\cap N=1$.
 \end{lem}
 \begin{proof}
     We proceed by induction on $|G|$.
If $|N|=2$, then $N\leq Z(G)$, and
$G$ is an abelian group of order $8$.
Since $o(zN)=4$, $G$ is not cyclic and
$o(z)=4$, it follows $N\cap \langle z\rangle=1$. 
Let $|N|>2$ and put $M=\mho_1(N)\langle z\rangle$.
If $N\leq \langle z\rangle$, then 
$G$ is a cyclic group, which is a contradiction.
Consequently,  $N\nleq \langle z\rangle$, and thus
  $N\cap \langle z\rangle\leq \mho_1(N)$.
  It follows that 
  $4=o(zN)=o(z\mho_1(N)).$
 Since  
 \begin{align*}
8=[G:\mho_1(N)]&=[G:M][M:\mho_1(N)]\\&=2[M:\mho_1(N)],
 \end{align*}
 we have $[M:\mho_1(N)]=4$. 
By the induction hypothesis, 
$M$ has an element $v$ of order $4$ such that $1=\langle v\rangle\cap \mho_1(N)=\langle v\rangle\cap N$.    
 \end{proof}

We need theses auxiliaries results to prove Theorem \ref{max}.

\begin{lem}\label{Berkovich}(Theorem 9.10 \cite{Ber}) If a group $G$ of order $p^m > p^3$ has a subgroup $M$ of order $p^{m-1}$ of maximal class, then $G$ is either of maximal class or $G/G' \cong C_{p^3}$.
    
\end{lem}

\begin{lem}\label{Janko}(Theorem 2.1 \cite{Jank}) Let $G$ be a metacyclic 2-group. Then
\begin{enumerate}
\item[(i)] $G$ contains exactly one involution if and only if $G$ is either cyclic or generalized quaternion.
\item[(ii)] $G$ contains more than three involutions if and only if $G$ is either dihedral or semidihedral.
\item[(iii)] All other metacyclic 2-groups contain exactly three involutions.
\end{enumerate}
    
\end{lem}

 \begin{thm}\label{max}

     Let $G$ be a finite non-abelian  $p$-group of $LC$-class  $k>1$. Then   for each  $ 0\leq i\leq  \ldots\leq  k-1$,   $\frac{LC_{i+1}(G)}{LC_i(G)}$ is cyclic if and only if  $k=2=p$, and  
  $G \in \{D_{2^n}, SD_{2^n}, Q_{2^m}\}$, where $n \geq 3$ and $m \geq 4$.
 \end{thm}
\begin{proof}
($\Leftarrow$) Follows immediately from Corollary \ref{dih}.

\noindent ($\Rightarrow$) Let $N=LC(G)$, and let $G$ be a minimal counterexample.
First suppose that $p>2$.
We claim that $\frac{G}{N}$ is a cyclic group.
If $k>2$, then 
by   minimality of $G$, $\frac{G}{N}\in \{D_{2^n}, SD_{2^n}, Q_{2^m}\}$, where $n \geq 3$ and $m \geq 4$, which is a contradiction.
Thus $k=2$, and $\frac{G}{N}$ 
is a 
cyclic group, as claimed.
Then $G=N\langle b\rangle$ for some $b\in G$.
Hence, $G'\leq N$ is cyclic, and it follows that $G$ is a regular $p$-group. Since all regular $p$-groups are $LCM$-groups, so $G=LCM(G)=N$,   which is a contradiction.

Therefore the only possible case is $p=2.$
  If $|G| = 8$, then clearly $G \cong D_8$. Assume $|G| > 8$.  
First suppose that $k=2$.
Then $G=N\langle b\rangle$, where $\frac{G}{N}=\langle bN\rangle$. 
Let $u\in \langle b\rangle$ 
 such that $o(uN)=2$.
If $N\langle u\rangle$ is abelian, then $u\in Z(G)\leq N$, which is a contradiction.
It follows  $N\langle u\rangle$ is not  abelian.
Hence $N\langle u\rangle$ is of maximal class. First suppose that $N\langle u\rangle\neq G$. Let $v\in \langle b\rangle$ such that $o(vN)=4$, and let $H=\langle v \rangle N$.
As $\frac{H}{H'}\ncong (C_2)^3$, by Lemma \ref{Berkovich}, $H$ is of maximal class. Then $H \in \{D_{2^n}, SD_{2^n}, Q_{2^m}, M_{2^m}\}$, where $n \geq 3$ and $m \geq 4$.
Let $F:=\langle f\rangle\leq H$ be subgroup of index two for some $f\in H$. Since $N\leq F$, we have $F\lhd G$,  and then by Lemma \ref{norm1},  $F\leq LCM(G)=N$, which is a contradiction.
Therefore $N\langle u\rangle=G$ is of maximal class. So $G \in \{D_{2^n}, SD_{2^n}, Q_{2^m}, M_{2^m}\}$  where $n \geq 3$ and $m \geq 4$. Since $G$ is a minimal counterexample
  $G\cong M_{2^m}$. As $M_{2^m}$ is an $LCM$-group, $LCM(G)=G$, which is a contradiction.

 Finally let $k\geq 3$.  By minimality of $G$, $\frac{G}{N}$ is a $LC$-nilpotent group of $LC$-class 2 and 
   $\frac{G}{N} \in \{D_{2^n}, SD_{2^n}, Q_{2^m}\}$, where   $n \geq 3$ and $m \geq 4$.
   Let $z\in G$ such that $o(zN)=|G|/2|N|$, and let $a\in \langle z\rangle$ such that $o(aN)=4$ whenever $\frac{G}{N}\in \{D_{2^n},SD_{2^n}\}$, otherwise  $o(aN)=8$.
   Let $u\in \langle z\rangle$ such that $o(uN)=2$.
We claim that $[u,N]=1$.  Suppose for a contradiction $[u,N]\neq 1$.
Then $H:=N\langle u\rangle$  is of maximal class.  Let $b\in \langle a\rangle$ such that $o(bN)=4$. 
Since $\frac{H\langle b\rangle}{(H\langle b \rangle)'}\ncong C_2^3$, by Lemma \ref{Berkovich}, $H\langle b \rangle=N\langle b\rangle$ is of maximal class. 
 It follows that 
  $N\leq  \langle f\rangle$ where $\langle f\rangle$  is cyclic maximal subgroup of $H\langle b\rangle$.
   Since  $uN=b^2N\subseteq  \langle f\rangle$, we deduce that  $[u,N]=1$, which is a contradiction. Therefore $[u,N]=1$, as claimed.
  Hence, $N\langle u\rangle=N\times \langle x\rangle$ for some $x\in N\langle u\rangle$ of order two. It follows that 
$\Omega_1(N\langle u\rangle)\cong C_2\times C_2$. 

 If $|\Omega_1(N\langle z\rangle)| > 4$, then by Lemma \ref{Janko} (ii), $N\langle z\rangle$ is either dihedral or semidihedral.  
Therefore  $N\langle z^2\rangle \leq \langle f \rangle$, where $\langle f \rangle$ is the cyclic maximal subgroup of $N\langle z\rangle$.  
Since $\langle f \rangle$ in $N\langle z\rangle$ is characteristic, we have $\langle f \rangle \lhd G$.  
It follows from Lemma  \ref{norm1} that  $f \in N$, which leads to a contradiction.
Thus, $|\Omega_1(N\langle z\rangle)| = 4$.  
It follows that $K := \Omega_1(N\langle z\rangle) \cong C_2 \times C_2$.  
According to Lemma \ref{sub4}, there exists $v \in N\langle z\rangle$ such that  
$o(v) = 4$ and $\langle v \rangle \cap N = 1$.

Let $P=N\langle v\rangle\langle w\rangle$. We have $\frac{N}{\mho_1(N)}\leq Z(\frac{P}{\mho_1(N)})$, as 
$|\frac{N}{\mho_1(N)}|=2$.
A GAP computation, \cite{Gap11}, shows that in the set of all non-abelian  finite  groups of order $16$ with  center of order equal to or greater than $4$, the derived subgroup has order two.
Therefore   $\varPhi(P)\leq \mho_1(N)\langle x\rangle$, and thus $[P:\varPhi(P)]=8$.
It follows that
$\mho_1(N)\langle v\rangle \langle w\rangle$ is a   proper subgroup of $P$.
Next we repeat the above argument for 
$\mho_i(N)$ for all $i=1,\ldots,m-1$, where $|N|=2^m$.
Then $$F:=\mho_{m-1}(N)\langle v\rangle \langle w\rangle=\Omega_{1}(N)\langle v\rangle \langle w\rangle$$ is a subgroup of $P$. 
Since $\Omega_1(N)\leq Z(P)$, we have 
$\Omega_1(\Omega_1(N)\langle v\rangle)=K$. 
Therefore   $T:=K\langle w\rangle$ is a subgroup of $G$ of order $8$.
As, $[w,x]\neq 1$,  we conclude that 
$T\cong D_8$.
Hence $F=\langle w,v\rangle$.
Since $T\langle v\rangle=F$, and 
$\frac{F}{F'}\ncong C_2^3$,  $F$ is of maximal class.
Let $\langle f\rangle$ be the cyclic maximal subgroup of $F$.
Then $N\langle x\rangle \leq \langle f\rangle$, which is our final contradiction.
\end{proof}
  \begin{thm}\label{meta}
     Let $G$ be a finite group. Then $G$ is an $LC$-nilpotent group of $LC$-class $k$ such that $\frac{LC_{i+1}(G)}{LC_i(G)}$ is cyclic for all $i = 0, \ldots, k-1$ if and only if  $k=2$ and $G = LC_1(G)H$, where $H$ is a cyclic subgroup of $G$.  
   \end{thm}
 \begin{proof}
($\Leftarrow$) This direction is straightforward. 
 
\noindent ($\Rightarrow$) Let $N = LC_1(G)$. Since $N$ is a nilpotent group, $N\leq Fit(G)$.  We proceed by induction on $|G|$. If $k = 2$, then
$\frac{G}{N}=\frac{LC_2(G)}{N}=\langle yN\rangle$ for some $y\in G$. It follows that 
$G = N \langle y \rangle$.

Next suppose $k > 2$. By the induction hypothesis, $\frac{G}{N} = \frac{LC_2(G)}{N} \frac{H}{N}$, where $\frac{LC_2(G)}{N} = \langle vN \rangle$ and $\frac{H}{N} = \langle hN \rangle$ are cyclic groups. Then $k = 3$.
 
Let $P\in \operatorname{Syl}_p(\operatorname{Fit}(G))$ and let $R\in \operatorname{Syl}_p(LC_2(G))$.
There exists $r\in LC_2(G)$ such that $R=(N\cap R)\langle r\rangle$.
Then $R'$ is cyclic, thus $R$ is a regular group.
By Corollary \ref{42}, $P \leq N$. 
It follows that $\operatorname{Fit}(G)=\langle b\rangle \times S$ where $S$ is Sylow $2$-subgroup of $\operatorname{Fit}(G)$ and $o(b)$ is an odd  number. We consider the following two cases:

{\bf Case 1.}
 If  $S$ is a cyclic group,    then 
 by Lemma \ref{norm1},
 $S\leq N$, so $\operatorname{Fit}(G)=N$ 
 Therefore 
$\frac{G}{C_G(\operatorname{Fit}(G))}=\frac{G}{\operatorname{Fit}(G)}$ is isomorphic to a subgroup of abelian subgroup of $\operatorname{Aut}(\operatorname{Fit}(G))$.
Hence $LC(\frac{G}{\operatorname{Fit}(G)})=\frac{G}{\operatorname{Fit}(G)}$, which is a contradiction.

{\bf Case 2.}
Therefore $S$ is not a cyclic group. 
 If $Q$ is an $LCM$-group, then by Corollary \ref{42}, we have $S \leq N$.  
Thus, $S$ is a cyclic group, which leads to a contradiction.  
Consequently, $Q$ is not an $LCM$-group.

 If $LC_2(Q)\neq Q$, then 
 $LC_3(Q)=Q$, as $LC_3(G)=G$.
 Also, $Q=\langle v_1\rangle \langle v_2\rangle\langle v_3\rangle$ where $v_i\in LC_i(G)$ for $i=1,2,3$ and $v_3\not\in LC_2(G)$. 
It follows from  $S\leq \langle v_1\rangle \langle v_2\rangle$ that $S$ is a  metacyclic group.
If $|\Omega_1(S)|>4$, then by Lemma \ref{Janko} (ii),  $S$ is of maximal class, thus
 $Z(S)$ is a cyclic group. 

By Lemma \ref{norm1},
 $Z(S)\leq N$.
 Since $G$ is a solvable group, $C_G(Fit(G))\leq Z(Fit(G))$.
 As $S$ is a $2$-group of maximal class and $Fit(G)=\langle b\rangle \times S$, we conclude that  $C_G(Fit(G))=C_G(b)\times C_G(Z(S))$.
 Therefore 
$\frac{G}{C_G(\langle b\rangle \times Z(S))}=\frac{G}{\operatorname{Fit}(G)}$ is isomorphic to a subgroup of abelian subgroup of $\operatorname{Aut}(\langle b\rangle \times Z(S))$.
Hence $LC(\frac{G}{\operatorname{Fit}(G)})=\frac{G}{\operatorname{Fit}(G)}$,
so $k=2$, which leads to a contradiction.

Therefore, $|\Omega_1(S)| = 4$. 
Let $\varphi(S) = D$, then 
\[
\operatorname{Fit}\left(\frac{G}{D}\right) = \frac{\operatorname{Fit}(D)}{D}.
\]
Also, since $S$ is a metacyclic group, we have 
\[
\frac{S}{D} \cong C_2 \times C_2.
\]

First, suppose that 
\[
\frac{G/D}{C_{G/D}\left(\frac{S}{D}\right)} \cong S_3.
\]
Let $yD, vD \in \frac{G}{D}$ such that $o(yD) = 3$ and $o(vD) = 2$. As,  $|\frac{S}{D} \langle yD, vD \rangle|=24$, $\frac{S}{D}\lhd \frac{S}{D} \langle yD, vD \rangle$ and  $\frac{S}{D} \langle yD\rangle\lhd \frac{S}{D} \langle yD, vD \rangle$, we conclude that    
\[
\frac{S}{D} \langle yD, vD \rangle \cong S_4
\]
is a $2$-Frobenius group of type $(2, 3, 2)$, which implies that $G$ is not an $LC$-nilpotent group, which is a contradiction.

Hence, 
\[
\left|\frac{G/D}{C_{G/D}\left(\frac{S}{D}\right)}\right| \in \{1, 2, 3\}.
\]

We also have 
\[
\frac{G}{\operatorname{Fit}(G)} \cong \frac{G/D}{\operatorname{Fit}(G)/D} = \frac{G/D}{C_{G/D}\left(\frac{\operatorname{Fit}(G)}{D}\right)}.
\]

Since 
\[
C_{G/D}\left(\frac{\operatorname{Fit}(G)}{D}\right) = C_{G/D}(bD) \cap C_{G/D}\left(\frac{S}{D}\right),
\]
and there exists a monomorphism from $\frac{G}{Fit(G)}$ into the cyclic group 
\[
\frac{G/D}{C_{G/D}(bD)} \times \frac{G/D}{C_{G/D}\left(\frac{S}{D}\right)},
\]
we deduce that $\frac{G}{\operatorname{Fit}(G)}$ is a cyclic group. Hence, $G' \leq \operatorname{Fit}(G)=NS$. 
Since $\frac{NS}{N}$ is a cyclic group, by Lemma \ref{norm1}, we have 
\[
\frac{NS}{N} \leq LC\left(\frac{G}{N}\right).
\]
Hence, $\frac{NS}{N}$ is cyclic.

If $Q \leq LC_2(G)$, then $\frac{QN}{N} = \langle uN \rangle$ is cyclic, implying 
\[
Q = (N \cap Q)\langle u \rangle.
\]  
Let $v \in Q$ such that $o(vN) = 2$. Since $v \notin LC(G)$, there exists a $2$-element $w \in G$ such that $o(vw) \nmid lcm(o(v), o(w))$. Let $M = (N \cap Q)\langle v, w \rangle$. 

Since $\frac{Q}{N \cap Q}$ is cyclic and $N \cap M \leq LC(M)$, it follows that $\frac{M}{LC(M)}$ is cyclic. By the induction hypothesis, $Q$ is of maximal class, so $S$ is cyclic, which leads to a contradiction.

Thus, $Q \nleq LC_2(G)$. Let $M = NSQ$. Since $G' \leq NS \lhd G$, we conclude that $M \lhd G$. Hence, $M$ contains all $2$-elements of $G$. It follows that 
\[
Q \cap LC(M) = Q \cap N \quad \text{and} \quad Q \cap LC_2(M) = Q \cap LC_2(G),
\]
as $S \nleq N$. Therefore, $LC(M) = N$ and $LC_2(M) = LC_2(G) \cap M$.

If $M \neq G$, then by the induction hypothesis, $M$ is a $2$-group of maximal class, so $S$ is cyclic,  which is a contradiction. Thus, $M = G$.

By Theorem \ref{max}, we have 
\[
\frac{QN}{N} \in \{D_{2^n}, SD_{2^n}, Q_{2^m}\}, \quad \text{for } n \geq 3 \text{ and } m \geq 4.
\]

Let $\langle fN \rangle$ be the maximal cyclic subgroup of $\frac{QN}{N}$. Then $f \in S$, and thus $S$ is of maximal class, leading to a contradiction. The proof is complete.
 \end{proof}

\subsection*{Funding}
The first author was partially supported by FAPDF, Brazil. The third author was partially supported by DPI/UnB, NSF of China (Grant No. 12161035) and FEMAT Proc. 054/2022, FAPDF, Brazil.

\subsection*{Acknowledgements}
We thank Pavel Shumyatsky for pointing out a mistake in the statement of  Theorem \ref{lc5}. We also thank the referees for their careful reading and suggestions.

\medskip
{\it Author's Addresses:\\}
\small
Mohsen Amiri\\
Universidade Federal de Uberl\^andia,\\
Faculdade de Matem\'atica - UFU\\  
38.408-902, Uberl\^andia-MG, Brazil\\

Iryna Kashuba\\
Shenzhen International Center for Mathematics,\\
Southern University of Science and Technology\\  
518055, Shenzhen, China\\

Igor Lima,\\
Universidade de Bras\'ilia,\\
Departamento de Matem\'atica\\
70910-900 Bras\'ilia - DF, Brazil

\end{document}